\documentclass[12pt, a4paper]
{amsart}

\usepackage{amscd,amssymb,mathtools}
\usepackage{tikz-cd}
\usepackage[shortlabels]{enumitem}
\mathtoolsset{showonlyrefs}

\allowdisplaybreaks

\usepackage{color, url, hyperref}

\usepackage{soul}
\usepackage{xcolor}
\usepackage{bbm}

\hypersetup{colorlinks=true, linkcolor=blue, linkbordercolor=blue, citecolor=blue, citebordercolor=blue, linktocpage=true}

\newcommand{\norm}[1]{\left\|#1\right\|} 

\usepackage{pgffor}

\newcommand{\defmathbb}[1]{%
    \expandafter\gdef\csname#1\endcsname{\mathbb{#1}}%
}
\newcommand{\defmathcal}[1]{%
    \expandafter\gdef\csname#1#1\endcsname{\mathcal{#1}}%
}
\foreach \char in {A,B,...,Z} {%
    \expandafter\defmathbb\char%
    \expandafter\defmathcal\char%
} 

\newtheorem{thm}{Theorem}[section]
\newtheorem{lemma}[thm]{Lemma}
\newtheorem{prop}[thm]{Proposition}
\newtheorem{coro}[thm]{Corollary}

\newtheorem{introthm}{Theorem}

\theoremstyle{definition}
\newtheorem{defin}[thm]{Definition}

\newtheorem{remark}[thm]{Remark}
\newtheorem{notation}[thm]{Notation}

\theoremstyle{remark}

\numberwithin{equation}{section}

\title{Graph C*-algebras are singly generated}

\newcommand{\Esink}{E^0_{\mathrm{sink}}}
\newcommand{\Eboundary}{E^0_{\mathrm{bdry}}}
\newcommand{\Eint}{E^0_{\mathrm{int}}}

\newcommand{\edgesink}{E^1_{\mathrm{sink}}}
\newcommand{\edgebdy}{E^1_{\mathrm{bdry}}}
\newcommand{\edgeint}{E^1_{\mathrm{int}}}

\newcommand{\Cstar}{\ensuremath{C^*}}

\newcommand{\twoone}{\ensuremath{\mathrm{II}_{1}\,}}
\newcommand{\bT}{\ensuremath{\mathbb{T}}}
\newcommand{\bF}{\ensuremath{\mathbb{F}}}
\newcommand{\cZ}{\ensuremath{\mathcal{Z}}}

\newcommand{\newY}{\ensuremath{Y}}

\author[Curda]{Jakub Curda}
\address{Jakub Curda\\ Mathematical Institute\\ University of Oxford\\ Oxford, OX2 6GG\\
United Kingdom}
\email{jakub.curda@maths.ox.ac.uk}

\author[Gonzales]{Julian Gonzales}
\address{Julian Gonzales\\ School of Mathematics and Statistics  \\
University of Glasgow\\ University Place\\ Glasgow, G12 8QQ\\ United Kingdom}
\email{julian.gonzales@glasgow.ac.uk}

\author[Wu]{Victor Wu}
\address{Victor Wu\\ School of Mathematics and Statistics \\ University of Sydney \\ Camperdown, NSW 2050 \\ Australia}
\email{victor.wu1@sydney.edu.au}

\begin{document}

\maketitle

\begin{abstract}
    We show that the $C^*$-algebra of a countable directed graph is singly generated. As a consequence, any $C^*$-algebra generated by a countable family of projections and partial isometries satisfying Cuntz-Krieger relations is singly generated.
\end{abstract}

\section{Introduction}

The question of which operator algebras are generated by a single element, or equivalently two self-adjoint elements, goes at least back to Kadison's list of problems from 1967 (see \cite{Ge03}). 
He asked the question for separably acting \twoone factors, or more generally, separably acting von Neumann algebras. 
Kadison's problem has been solved for all factors not of type \twoone (\cite{Pearcy62,Wogen69}), and many classes of \twoone factors are also known to be singly generated (see for example \cite{Popa85,GePopa98,Sri}).
The problem remains open in its full generality - for example, it is open for the free group factors $L(\bF_n)$ when $n\geq3$.

In contrast to the von Neumann algebraic setting, it is too naive to ask if all separable \Cstar-algebras are singly generated - one does not need to look far to find counterexamples. 
For unital commutative \Cstar-algebras, it is known that $C(X)$ is singly generated if and only if $X$ embeds into the complex plane. 
Noncommutative counterexamples are also abundant. Algebras of the form $C(X)\otimes M_n$ are not singly generated provided the matrix dimension $n$ is sufficiently small compared to the dimension of $X$ (see \cite{Nagisa04} for instance). 
Since the quotient of a singly generated \Cstar-algebra remains singly generated,  more counterexamples can be constructed from extensions of these. For example, the full group \Cstar-algebra $\Cstar(\bF_2)$ of $\bF_2$ is not singly generated since it quotients onto $C(\bT^2)$.

It is therefore unsurprising that the generator problem for \Cstar-algebras is often restricted to algebras that lack finite-dimensional representations (see \cite[Question~1.3]{TW14} for instance), or sometimes even to simple algebras.
To the best of our knowledge, the only known obstruction to single generation is the presence of homogeneous quotients that are not singly generated.

A natural class of \Cstar-algebras that do not admit problematic quotients is the class of graph \Cstar-algebras; see the decomposition described below. 
In this paper we show that for any countable directed graph $E$, its graph \Cstar-algebra $\Cstar(E)$ is singly generated. 

\begin{introthm}
\label{thm}
Let $E$ be a countable directed graph. 
Then $\Cstar(E)$ is singly generated.
\end{introthm}

First introduced in \cite{KPRR} (see also \cite{Fowler-Laca-Raeburn}), graph \Cstar-algebras have proven to be a class of algebras which often allow for explicit calculations (\cite{RuizSimsTomforde15,EilersKatsuraTomfordeWest16,EilersRestorffRuizSorensen21}), direct study of \Cstar-algebraic properties (\cite{EilersGabeKatsuraRuizTomforde20}), and testing of conjectures (\cite{SuriTomforde17}).
Even though these algebras are rarely simple, loosely speaking, they arise as iterated extensions of AF algebras, Kirchberg algebras and algebras of the form $C(\bT)\otimes M_n$ (\cite{BatesPaskRaeburnWojciech00}), all of which are known to be singly generated (see the paragraph below). 
Since single generation does not pass to extensions (see \cite{T21}), Theorem~\ref{thm} cannot be deduced directly from this fact.
Nevertheless, this suggests that graph \Cstar-algebras are a natural class on which to consider the generator problem.
Observe that Theorem~\ref{thm} does not extend to \Cstar-algebras of higher-rank graphs in full generality - it is known that a $k$-graph can have \Cstar-algebra isomorphic to $C(\bT^k)$ (see \cite[Example~1.7 (iii)]{Kumjian-Pask}).

Let us give a brief overview of known single generation results in the \Cstar-setting. 
Early results (\cite{Topping68,OlsenZame76,LiShen08}) focused on algebras with suitable decomposition properties.
This culminated with the work \cite{TW14} which showed that, in particular, any unital separable and $\cZ$-stable \Cstar-algebra is singly generated.
We also note that stabilisations of unital separable \Cstar-algebras are singly generated (\cite[Theorem~8]{OlsenZame76}) and so are all AF algebras (\cite{T21}).
In a different direction, an unpublished result, due to Kirchberg, states that any unital separable \Cstar-algebra containing two isometries with orthogonal ranges is singly generated (see \cite[p.~138]{Nagisa04} for the statement and \cite[Theorem~2.4]{TW14} for a sketch of the proof in a special case). 
We also highlight that one of the major obstacles in proving single generation is a severe lack of permanence properties: single generation is not necessarily preserved by inductive limits, extensions, nor stable isomorphisms.
In fact, the stabilization of a unital separable \Cstar-algebra is singly generated by \cite[Theorem~8]{OlsenZame76}.

The class of graph \Cstar-algebras contains algebras that are stable and/or \cZ-stable (see \cite{Tomforde04,Faurot25}).
For the graph algebras falling into the classes above, our methods give an alternative proof for their single generation. 
The class also contains algebras covered by Kirchberg's result, although those are necessarily unital.
Note that the generator problem is often restricted to unital \Cstar-algebras, since a \Cstar-algebra is singly generated if and only if its minimal unitization is. 
We do not make any unitality assumptions in this paper - indeed, a graph has unital \Cstar-algebra if and only if it has finitely many vertices.


Let us now illustrate the basic idea of our proof in the special case of a finite graph with no sinks.
Let $E$ be such a graph with edges $e_1,\dots,e_m$. Writing $s_i$ for the corresponding partial isometries, note that $\Cstar(E) = \Cstar(s_1,\dots,s_m)$. 
It turns out that if $E$ contains no sinks, the element $d= \left( \sum s_i \right)\zeta$ is a generator for $\Cstar(E)$ for a suitable choice of weight $\zeta$ belonging to the commutative subalgebra generated by all of the vertex projections and range projections of edge partial isometries.
Using orthogonality of the ranges of the $s_i$, one finds that $d^*d$ is a weighted sum of the range projections, and hence they all belong to $\Cstar(d)$. 
Finally, multiplying $d$ on the left by $s_is_i^*$ shows that $s_i$ also belongs to $\Cstar(d)$. 
We note that the idea of summing scalar multiples of isometries already appeared in the proof sketch of Kirchberg's result in \cite[Theorem~2.4]{TW14}. 

It turns out that the strategy described above makes no use of the unit and so extending to countable graphs without sinks is straightforward (see Corollary~\ref{corollary:no-sinks}). The main obstacles are the sinks in $E$. More specifically, in the presence of sinks receiving infinitely many edges, the graph \Cstar-algebra is no longer generated by the edge partial isometries. We adjust the generator to be a sum of three terms, each designed to handle a different portion of the graph. The first captures the sinks and the edges terminating at them, the second captures all the edges ``sufficiently far" from the sinks, and the third takes care of the remaining edges and helps glue the first two terms together. A careful double induction argument lets us recover each edge partial isometry and vertex projection from the generator.

The paper is organised as follows: Section \ref{sec:prelims} recalls some basic facts about spectral theory and graph \Cstar-algebras, while Section \ref{sec:no-sink} treats the special case of a graph $E$ with no sinks. In Section \ref{sec:g} we construct the candidate generator for $\Cstar(E)$ and then verify that it is indeed a generator in Section \ref{sec:proof}.










\section*{Funding}
\begin{sloppypar}
The authors would like to thank the Isaac Newton Institute for Mathematical Sciences, Cambridge, for support and hospitality during the programme {\bf{Topological groupoids and their $C^{*}$-algebras}}, where work on this paper was undertaken. This work was supported by EPSRC grant EP/V521929/1. J.C. was supported by EPSRC grant EP/X026647/1. J.G. was supported by EPSRC grant EP/T517896/1. For the purpose of Open Access, the authors have applied a CC BY public copyright licence to any Author Accepted Manuscript (AAM) version arising from this submission.
\end{sloppypar}

\section*{Acknowledgments}

The authors would like to thank Brian Chan for helpful discussions about topics related to this paper.

\section{Preliminaries and basic facts - spectra}
\label{sec:prelims}

To fix notational conventions, let us recall the definition of a directed graph and its associated \Cstar-algebra.

\begin{defin}
    A \emph{directed graph} $E = (E^0, E^1, r, s)$ is a quadruple consisting of countable \emph{vertex} and \emph{edge} sets $E^0$ and $E^1$, along with \emph{range} and \emph{source} maps $r,s\colon E^1 \to E^0$.
\end{defin}

\begin{defin}[{\cite[Definition~1]{Fowler-Laca-Raeburn}}]\label{def:graph}
    Let $ E = (E^0, E^1, r, s)$ be a directed graph. A \emph{Cuntz-Krieger $E$-family} in a \Cstar-algebra $A$ consists of a family $\{p_v \colon v \in E^0\} \subseteq A$ of mutually orthogonal projections and a family $\{s_e \colon e \in E^1\} \subseteq A$ of partial isometries with mutually orthogonal ranges such that
    \begin{enumerate}[label=(\roman*)]
        \item $s_e^*s_e = p_{s(e)}$ for all $e \in E^1$;
        \item $s_e s_e^* \le p_{r(e)}$ for all $e \in E^1$; and
        \item $p_v = \sum_{r(e)=v} s_e s_e^*$ whenever $0 < |r^{-1}(v)| < \infty$.
    \end{enumerate}
    The \emph{graph \Cstar-algebra} $\Cstar(E)$ is the universal $\Cstar$-algebra generated by a Cuntz-Krieger $E$-family.
\end{defin}

We use the convention that edges and their associated partial isometries go in the same direction, following \cite{Raeburn}. This is in contrast to \cite{Fowler-Laca-Raeburn} and \cite{Rangaswamy-Tomforde}, for example, where the opposite convention is adopted.    
Universality of $\Cstar(E)$ will not be important for showing that it is singly generated; we prove directly that any \Cstar-algebra generated by a Cuntz-Krieger $E$-family is singly generated.

Let us introduce some notation for edges and vertices in a graph. The terms \emph{interior} and \emph{boundary} will describe how far an edge or vertex is from a sink. The generator we construct in Section~\ref{sec:g} will have three components coming from interior edges, boundary edges and sinks.

\begin{notation}
    Let $E = (E^0, E^1, r, s)$ be a directed graph. A vertex $v \in E^0$ is called a \emph{sink} if $s^{-1}(v) = \emptyset$ (that is, if $v$ does not emit any edges); we denote by $\Esink \subseteq E^0$ the set of sinks. We shall call a vertex $v \in E^0$ a \emph{boundary vertex} if it is not a sink and $r(s^{-1}(v)) \subseteq \Esink$ (that is, if all edges emitted from $v$ terminate at a sink); we denote by $\Eboundary \subseteq E^0$ the set of boundary vertices. We shall call all other vertices \emph{interior} vertices; we denote by $\Eint \subseteq E^0$ the collection of interior vertices (so $E^0 = \Eint \sqcup \Eboundary \sqcup \Esink$).

    We call an edge $e \in E^1$ a \emph{sink edge} if $r(e) \in \Esink$, a \emph{boundary edge} if $r(e) \in \Eboundary$, and an \emph{interior edge} if $r(e) \in \Eint$. We denote by $\edgesink \subseteq E^1$ the set of sink edges, by $\edgebdy \subseteq E^1$ the set of boundary edges, and by $\edgeint \subseteq E^1$ the set of interior edges.
\end{notation}

\begin{notation}
    Let $E$ be a directed graph. For vertices $v,w \in E^0$ we denote by $p_v, q_w \in \Cstar(E)$ the associated projections, and for edges $e,f \in E^1$ we denote by $s_e, t_f \in \Cstar(E)$ the associated partial isometries. We call $s_e s_e^*$ the \emph{range projection} associated with the edge $e$. From now on, $w$ will exclusively be used to denote a sink, and $f$ will exclusively be used to denote a sink edge.
\end{notation}

\subsection{Spectra} In this subsection we record two basic facts about spectra. The first is a general fact about orthogonal normal elements in a $\Cstar$-algebra. The second states that certain self-adjoint elements in a graph $C^*$-algebra have countable spectrum, and will be crucial for our construction of the generator in Section~\ref{sec:g}.

If $A$ and $B$ are singly generated \Cstar-algebras and $A$ is unital\footnote{Note that the direct sum of two non-unital singly generated \Cstar-algebras need not be singly generated. Consider $C_0(X) \oplus C_0(X)$ where $X$ is the punctured disk, for example.}, then the direct sum $A\oplus B$ is also singly generated\footnote{Let $a$ and $b$ be their respective generators and without loss of generality assume that $\norm{a},\norm{b} < \frac{1}{2}$. After taking real and imaginary parts, an application of functional calculus reveals that $(1_A+ i1_A + a, b)$ is a generator for $A\oplus B$.}.
We will need the following simple variation of this fact in the commutative setting, where we drop the assumption that $A$ is unital and instead ask for control over the spectra of the generators $a \in A$ and $b \in B$.

\begin{lemma}
\label{lemma:disjoint-spectra}
    Let $A$ be a \Cstar-algebra and $a,b\in A$ two orthogonal normal elements with
    \begin{equation}\label{eq:disjoint-spectra}
        \sigma(a)\setminus\{0\} \cap \sigma(b)\setminus\{0\} =\emptyset\,.
    \end{equation} 
    Then $a,b\in\Cstar(a+b)$.
\end{lemma}

\begin{proof}
    By orthogonality, 
    \begin{align}
        \Cstar(a,b)&\cong \Cstar(a)\oplus\Cstar(b)\\
        &\cong C_0\big(\sigma(a)\setminus \{0\}\big)\oplus C_0\big(\sigma(b)\setminus \{0\}\big) \quad\\
        &\cong C_0\big( (\sigma(a) \setminus \{0\}) \sqcup (\sigma(b) \setminus \{0\} )\big)\,.
    \end{align}
    Since $\sigma(a)$ and $\sigma(b)$ are compact, it follows from \eqref{eq:disjoint-spectra} that the disjoint union $(\sigma(a) \setminus \{0\}) \sqcup (\sigma(b) \setminus \{0\} )$ is homeomorphic to $U \coloneq \big(\sigma(a)  \cup \sigma(b)\big) \setminus \{0\} \subseteq \C$. Under the identification $\Phi \colon C^*(a,b) \xrightarrow{\cong} C_0(U)$ we find
    \begin{align}
        \Phi(a)(z)&=\begin{cases}
            z\quad\text{if }z\in\sigma(a)\\
            0\quad\text{if }z\in \sigma(b)
        \end{cases}\\
        \Phi(b)(z)&=\begin{cases}
            0\quad\text{if }z\in\sigma(a)\\
            z\quad\text{if }z\in \sigma(b)
        \end{cases}\\
        \Phi(a+b)(z)&=z \quad \text{for all } z \in U\,.
    \end{align}
    Since $\Phi(a+b)$ separates points and vanishes nowhere in $U$, it generates $C_0(U)$ as a $C^*$-algebra (by the locally compact version of the Stone-Weierstrass theorem). Hence $a,b \in C^*(a+b)$ as desired.
\end{proof}

In order to apply Lemma \ref{lemma:disjoint-spectra} in our setting, we will need the following to control the spectra of certain self-adjoint elements in graph \Cstar-algebras.

\begin{lemma}
\label{lemma:countable-spectrum}
    Let $E$ be a directed graph.
    Let $B$ be the $C^*$-subalgebra of $C^*(E)$ generated by $\Esink \cup \edgesink \cup \edgebdy$.
    Then every self-adjoint element of $B$ has countable spectrum.
\end{lemma}

\begin{proof}
    Let $F$ be the subgraph of $E$ consisting only of the edges $\edgesink \cup \edgebdy$ along with their ranges and sources, together with $\Esink$. Since $F$ is entrance complete (in the sense that for each edge $e \in F^1$, we have $r_E^{-1}(r_E(e)) \subseteq F^1$), the family
    \begin{equation}
        \{p_v : v \in \Esink \cup \Eboundary\} \cup \{s_e : e \in \edgesink \cup \edgebdy\}
    \end{equation}  
    is a Cuntz-Krieger $F$-family in $C^*(E)$, so there is a $*$-homomorphism $\varphi\colon C^*(F) \to B$. Clearly $\varphi$ is surjective; and since $F$ is acyclic, the Cuntz-Krieger uniqueness theorem (\cite[Theorem~2]{Fowler-Laca-Raeburn})\footnote{As noted earlier, this paper uses the opposite convention for the direction of edge partial isometries to that used in \cite{Fowler-Laca-Raeburn} and \cite{Rangaswamy-Tomforde}.} shows that $\varphi$ is injective and hence an isomorphism.

    Note that $F$ is acyclic and has countably many paths, since the maximum path length in $F$ is $2$. By \cite[Theorem~6.5]{Rangaswamy-Tomforde} $C^*(F)$ has an elementary composition series (of countable length). It then follows from \cite{Huruya} that every self-adjoint element of $\Cstar(F) \cong B$ has countable spectrum, as required. 
\end{proof}

\section{Graphs with no sinks}
\label{sec:no-sink}

In this section we present a short proof that the \Cstar-algebra of a graph with no sinks is singly generated. As in Section~\ref{sec:proof}, a key feature of graph \Cstar-algebras we use is that edge partial isometries have mutually orthogonal range projections. As remarked on in the introduction, a similar idea is used in the proof of \cite[Theorem~2.4]{TW14}.

The generator $d$ in this section is of the form $\left( \sum \epsilon_e {s_e} \right) \xi^{\frac{1}{2}}$, where $\sum \epsilon_e {s_e}$ is a weighted sum of all edge partial isometries, and $\xi$ is a positive element belonging to the (commutative) \Cstar-subalgebra generated by all vertex and range projections. The element $\xi$ is chosen so that the \Cstar-subalgebra generated by $d^*d$ contains all range projections $s_e s_e^*$. Range projections are pairwise orthogonal, so multiplying the generator $d$ on the left by $s_e s_e^*$ shows that $s_e$ also belongs to $\Cstar(d)$. Since the \Cstar-algebra of a graph with no sinks is generated by its edge partial isometries, it follows that $d$ generates $\Cstar(E)$.

The following lemma describes the construction of $\xi$, and will also be used in Section~\ref{sec:g}. Roughly speaking, the existence of an appropriate positive element $\xi$ is guaranteed because any commutative \Cstar-algebra generated by a family of commuting projections has 0-dimensional spectrum and, in particular, is therefore singly generated.

\begin{lemma}\label{lemma:funny-function}
    Let $S \subseteq E^1$ be a set of edges in a directed graph $E$. Let $V \coloneq s(S) \subseteq E^0$ be the set of sources of edges in $S$, and let $T \coloneq r^{-1}(V) \subseteq E^1$ be the set of all edges that have range in $V$. Define 
    \begin{equation}
        A \coloneq \Cstar(p_v, s_e s_e^* \colon v \in V, e \in T) \subseteq \Cstar(E)\,.
    \end{equation}
    
    Let $(\epsilon_e)_{e \in S}$ be positive real numbers satisfying $\sum_{e \in S} \epsilon_e < \infty$. There exists a positive element $\xi \in A$ of norm at most 1 for which
    \begin{equation}
    d \coloneq \left( \sum_{e \in S} \epsilon_e s_e \right)\xi^{\frac{1}{2}}
    \end{equation}
    satisfies the following:
    \begin{enumerate}[label=(\roman*)]
        \item $A = \Cstar(d^*d)$;
        \item for any $v \in V$, there exists $h_v \in A$ such that $p_v = \xi^{\frac{1}{2}} h_v$.
    \end{enumerate}
    Moreover, given any countable compact set $C \subseteq \C$, the element $\xi$ can be chosen so that $d^*d$ has spectrum disjoint from $C \setminus \{0\}$.
\end{lemma}

\begin{proof}
    Set $\mu_v \coloneq \sum_{\substack{e \in S \cap s^{-1}(v)}} \epsilon_e^2$ for each $v \in V$. Note that each $\mu_v$ is finite and strictly positive because the set $S \cap s^{-1}(v)$ is non-empty for every $v \in V$. Let $C \subseteq \C$ be countable and compact, so that for any $\varepsilon>0$ there exists an interval $[\theta, \theta'] \subseteq (0,\varepsilon)$ of non-zero length such that $C\cap [\theta,\theta']=\emptyset$. Let $\big\{[\theta_v,\theta'_v]\big\}_{v \in V} \subseteq (0, 1)$ be a set of pairwise disjoint intervals of non-zero length such that $C\cap [\theta_v,\theta'_v]=\emptyset$ and $\frac{\theta_v'}{\mu_v} \le 1$ for all $v \in V$. If $V$ is infinite, we also require that for any $\varepsilon >0$, there only exist finitely many $v \in V$ for which $\frac{\theta_v'}{\mu_v} > \varepsilon$.
    
    The algebra $A$ is commutative with 0-dimensional spectrum. Let $\Omega$ denote its spectrum, and let $\Psi \colon A \xrightarrow{\cong} C_0(\Omega)$ be the canonical isomorphism. For each $v \in V$, let $K_v \subseteq \Omega$ be the compact open subset satisfying $\Psi(p_v) = \mathbbm{1}_{K_v}$. Note that $\Omega$ is the disjoint union of the $K_v$. Let $\xi \in A$ be such that for each $v \in V$, the restriction $\Psi(\xi)|_{K_v}$ is injective and has range in $[\frac{\theta_v}{\mu_v}, \frac{\theta_v'}{\mu_v}]$\footnote{This is possible because any second-countable 0-dimensional space embeds into the interval.}. Note that $\xi$ has norm at most 1.
    
    Set
    $
    r \coloneq \sum_{e \in S} \, \epsilon_e^2 \, s_e^* s_e \in A
    $
    and note that $r\xi = d^*d$ by orthogonality of ranges of the partial isometries $s_e$. On each compact open set $K_v$, the function $\Psi(r)$ takes constant value $\mu_v$, and therefore $\Psi(d^*d)|_{K_v}$ is injective and takes values in $[\theta_v, \theta_v']$. The function $\Psi(d^*d)$ is therefore strictly positive and injective on $\Omega$, and hence generates $C_0(\Omega)$ as a $\Cstar$-algebra (by the locally compact version of the Stone-Weierstrass theorem). It follows that $A = C^*(d^*d)$. This proves (i). Note also that $\Psi(d^*d)$ has range disjoint from $C \setminus \{0\}$, and hence $d^*d$ has spectrum disjoint from $C \setminus \{0\}$.

    We now prove (ii). Let $v \in V$. The function $\Psi(\xi)$ has range in $[\frac{\theta_{v}}{\mu_{v}}, \frac{\theta_{v}'}{\mu_{v}}]$ when restricted to the compact open set $K_{v}$. Since $\frac{\theta_{v}}{\mu_{v}} >0 $, there exists $h_v \in A$ such that the function $\Psi(h_v) \in C_0(\Omega)$ satisfies $\Psi\big(\xi\big)^{\frac{1}{2}} \Psi(h_v) = \mathbbm{1}_{K_{v}}$. Therefore $\xi^{\frac{1}{2}}h_v = p_v$, as desired.
\end{proof}

The following is now an immediate consequence of the preceding lemma.

\begin{coro}\label{corollary:no-sinks}
    Let $E$ be a directed graph with no sinks. Then $C^*(E)$ is singly generated.
\end{coro}

\begin{proof}
    Let $S = E^1$ be the set of all edges in $E$. Let $V$ and $T$ be as in Lemma~\ref{lemma:funny-function}. Since $E$ has no sinks, we have that $V = E^0$ and $T = E^1$. Let $(\epsilon_e)_{e \in E^1}$ be any sequence of positive real numbers satisfying $\sum_{e \in E^1} \epsilon_e < \infty$, and let $A$, $\xi$ and $d$ be as in Lemma~\ref{lemma:funny-function}. Take arbitrary $e \in E^1$, and let $h \in A = \Cstar(d^*d)$ be such that $\xi^{\frac{1}{2}}h = s_e^* s_e$. Since $s_e s_e^* \in A = \Cstar(d^*d)$, we can use the definition of $d$ to deduce that
    \begin{equation}
    \epsilon_e s_e = s_e s_e^* \left( \sum_{e \in S} \epsilon_e s_e \right) = s_e s_e^*\left( \sum_{e \in S} \epsilon_e s_e \right)\xi^{\frac{1}{2}}h = s_e s_e^* d h
    \end{equation}
    belongs to $\Cstar(d)$. It follows that $\Cstar(d) = \Cstar(E)$ because $\Cstar(E)$ is generated by its edge partial isometries.  
\end{proof}

\section{A generator for $\Cstar(E)$}
\label{sec:g}

Let $E$ be a countable directed graph. In this section we construct an element $g\in \Cstar(E)$ which will turn out to be a generator of $\Cstar(E)$.

\begin{notation}\label{label}
    Let $\{w_1, w_2, \dots\} = \Esink$ denote the set of sinks in $E$. For each sink $w_m$, let $\{f_{m,1}, f_{m,2}, \dots\} = r^{-1}(w_m)$ denote the set of edges terminating at $w_m$. Note that $\{f_{m,n}\}_{m,n \ge 1}$ is equal to $\edgesink$, the set of all sink edges. 
    For each sink $w_m$ let ${q_m \coloneq q_{w_{m}} \in \Cstar(E)}$ denote the associated projection, and for each sink edge $f_{m,n}$ let $t_{m,n} \coloneq t_{f_{m,n}} \in \Cstar(E)$ denote the associated partial isometry.

    For each $v \in \Eboundary$, let $m(v) \in \N$ be the smallest positive integer for which there exists an edge $f$ with $s(f) = v$ and $r(f) = w_{m(v)}$. Write $V_\ell \coloneq \{ v \in \Eboundary \colon m(v) = \ell\}$ for each $\ell \in \N$. Note that $V_1, V_2, \dots$ are pairwise disjoint, and $\Eboundary$ is equal to their union. 

    Set $\newY \coloneq s(\edgebdy) \subseteq \Eint$, the set of (interior) vertices $y$ for which there exists a (boundary) edge $e \in \edgebdy$ with $s(e) = y$. For each $y \in \newY$, let ${\{e_{y, 1},e_{y, 2}, \dots \} = \edgebdy \cap s^{-1}(y)}$ be the set of all boundary edges with source equal to $y$. Note that $\{e_{y,n}\}_{y \in \newY, n \ge 1}$ is equal to $\edgebdy$, the set of all boundary edges. 
    For each boundary edge $e_{y,n}$ let $s_{y,n} \coloneq s_{e_{y,n}} \in \Cstar(E)$ denote the associated partial isometry.\footnote{Note that any of the sets $\newY$, $V_\ell$, $\{w_1, w_2, \dots\}, \ \{f_{m,1}, f_{m,2}, \dots\}$ may be finite or empty. Given $y \in \newY$, the set $\{e_{y, 1},e_{y, 2}, \dots \}$ may also be finite.}
\end{notation}

Let us comment on the labelling of edges and vertices. Our goal in Section~\ref{sec:proof} will be to show that all edge partial isometries and sink projections belong to $\Cstar(g)$, where $g$ denotes the candidate generator for $\Cstar(E)$. If a family of projections or partial isometries has been enumerated in Notation~\ref{label}, then we will use induction to show that all elements of this family belong to $\Cstar(g)$ (see Lemma~\ref{lemma:boundary-edges} and Proposition~\ref{prop:proposition-4-part2}).

The following lemma is immediate by the Cuntz-Krieger relations in Definition~\ref{def:graph}.
\begin{lemma}\label{lemma:orthogonality}
    We have the following relations for all integers $m,n,m',n',k \ge 1$, edges $e, e' \in E^1$, and vertices $y \in \newY = s(\edgebdy)$:
    \begin{align}
    s_e^* s_{e'} &= 0 &&\text{ if } e \neq e'\,,\\
    q_m &\perp s_e  &&\text{ if } e \in \edgebdy \cup \edgeint\,,\\
    q_m &\perp p_v  &&\text{ if } v \in \Eboundary \cup \Eint\,,\\    
    p_v t_{m,n} &= 0 &&\text{ if } v \in \Eboundary \cup \Eint\,,\\
    t_{m,n} t_{m',n'} &= 0\,, \\
    s_{y,k} s_{y,n}^* s_{y,n} &= s_{y,k}\,.
    \end{align}
    Moreover, 
    \begin{equation}
        q_{m'} t_{m,n} =
    \begin{cases}
        t_{m,n} \quad &\text{if } m = m'\\
        0 \quad &\text{otherwise}
    \end{cases}\,.
    \end{equation}
\end{lemma}

Choose a strictly decreasing sequence $(\delta_m)_{m=1}^\infty \subseteq (0, 1)$ such that $\sum_{m=1}^\infty \delta_m < \infty$. Let $(\alpha_{y,n})_{y \in \newY,n \ge 1} \subseteq (0,1)$ satisfy $\sum_{y \in \newY, n \ge 1} \alpha_{y,n} < \infty$ and be such that 
\begin{equation}
    \sum_{r(e_{y,n}) \in V_\ell} \alpha_{y,n} \le \delta_\ell - \delta_{\ell+1}
\end{equation}
for all $\ell$. This implies that
\begin{equation}\label{DecompOfa_less_delta}
   \sum_{\ell \ge m} \sum_{r(e_{y,n}) \in V_\ell} \alpha_{y,n} \le \delta_m \quad\quad \text{for all }m\,.
\end{equation}

Let the double sequence $(\gamma_{m,n})_{m,n=1}^\infty$ be such that $\sum_{m,n=1}^\infty (\gamma_{m,n}-\delta_m) < \infty$, and such that $(\gamma_{m,n})_{n=1}^\infty \subseteq (\delta_m, \delta_{m-1})$ is a strictly decreasing sequence for each $m \in \mathbb{N}$. Set ${\beta_{m,n} \coloneq \frac{\gamma_{m,n}-\delta_m}{2}}$ for each $m,n \in \N$. Note that $\sum_{m,n=1}^\infty \beta_{m,n}< \infty$. Define
\begin{align}
    a &\coloneq \sum_{\substack{y \in \newY \\ n \ge 1}} \alpha_{y,n+1} s_{y,n+1} s_{y,n}^*\label{eq:a}\,,\\
    b &\coloneq \sum_{m,n \ge 1} \beta_{m,n} t_{m,n}\label{eq:b}\,, \\
    c &\coloneq \sum_{m \ge 1} \delta_m q_m + \sum_{m,n \ge 1} (\gamma_{m,n}-\delta_m) t_{m,n}t_{m,n}^*\label{eq:c}\,.
\end{align}
Note that all sums converge absolutely in norm, so that $a,b$ and $c$ are well-defined elements of $\Cstar(E)$.
Let $C \subseteq \C$ denote the spectrum of $a^*a + (b+c)^* (b+c)$. By Lemma~\ref{lemma:countable-spectrum}, $C$ is a countable compact set. Let $S$ denote the set of edges $\edgeint \cup \{e_{y,1}\}_{y \in \newY}$, and let $(\epsilon_e)_{e \in \edgeint}$ be a sequence of positive real numbers such that $\sum_{e \in \edgeint} \epsilon_e < \infty$. Note that $\Eint$ is equal to the set of sources of edges in $S$. 
Define $A(E)$ to be the subalgebra 
\begin{equation}\label{AbelianSubalgebra}
    A(E) \coloneq \Cstar(p_v,s_e s_e^* \colon v \in \Eint, e \in \edgeint).
\end{equation}
Applying Lemma~\ref{lemma:funny-function} to the set $S$ yields a positive element $\xi \in A(E)$ of norm at most 1 such that
\begin{equation}\label{eq:d}
d \coloneq \left( \sum_{y \in \newY} \alpha_{y,1} s_{y,1} + \sum_{e \in \edgeint} \epsilon_e s_e \right)\xi^{\frac{1}{2}}
\end{equation}
satisfies the following:
\begin{enumerate}[label=(\roman*)]
    \item $A(E) = \Cstar(d^*d)$;
    \item for any $e \in S$, there exists $h \in A(E)$ such that $s_e^*s_e = \xi^{\frac{1}{2}} h$;
    \item the spectrum of $d^*d$ is disjoint from $C \setminus \{0\}$.
\end{enumerate}
We now define the candidate generator $g$ for the graph $C^*$-algebra $C^*(E)$ by
\begin{equation}\label{generator}
    g \coloneq a+b+c+d    \,.
\end{equation}

\section{The proof}
\label{sec:proof}
In this section, we show that the element $g$ constructed in Section \ref{sec:g} generates $\Cstar(E)$. The strategy uses induction, and goes broadly as follows. Our goal is to show that each edge partial isometry $s_e$ and sink projection $q_i$ belongs to $\Cstar(g)$. Once we have shown that $s_e$ (resp. $q_i$) belongs to $\Cstar(g)$, we subtract all terms involving $s_e$ (resp. $q_i$) from $g$ to obtain a simplified generator (see \eqref{eq:g11} and \eqref{eq:g_NM}). The coefficients in the generator $g$ are carefully arranged so that all edge partial isometries and sink projections can be extracted via an induction procedure using these simplified generators (see Lemma~\ref{lemma:proposition-4-part1} and Proposition~\ref{prop:proposition-4-part2}).

In Subsection~\ref{subsec:interior} it is shown that all partial isometries associated with interior edges belong to $\Cstar(g)$. Similarly, Subsection~\ref{subsec:boundar} deals with boundary edges, and Subsection~\ref{subsec:sinks} deals with sinks and sink edges.


\subsection{Interior edges}
\label{subsec:interior}

In the following proposition, we mirror the proof of Corollary~\ref{corollary:no-sinks} and show that all partial isometries associated with interior edges belong to $\Cstar(g)$.

\begin{prop}
\label{proposition:interior-edges}
     Let $g \in C^*(E)$ be the candidate generator from \eqref{generator}, and let $A(E)$ be as in \eqref{AbelianSubalgebra}. Then $A(E) \subseteq C^*(g)$. Consequently, all partial isometries associated with interior edges belong to $\Cstar(g)$. In particular, $\left( \sum_{e \in \edgeint} \epsilon_e s_e \right)\xi^{\frac{1}{2}}$ belongs to $\Cstar(g)$.
\end{prop}

\begin{proof}
    By orthogonality of range projections of edges and Lemma~\ref{lemma:orthogonality}, we have that\footnote{In (\ref{eq:a}), the sum defining $a$ contains terms of the form $\alpha_{y,k}s_{y,k}s^*_{y,k-1}$ for $k \ge 2$, so $a^*d$ indeed vanishes.} ${a^*d = a^*(b+c) = d^*(b+c) = 0}$. Hence
    \begin{equation}
    g^*g = a^*a + (b+c)^*(b+c) + d^*d\,.
    \end{equation}
    The positive elements $a^*a + (b+c)^*(b+c)$ and $d^*d$ are orthogonal, and, away from 0, they have disjoint spectra by (iii) above. Therefore, $d^*d$ belongs to $\Cstar(g)$ by Lemma~\ref{lemma:disjoint-spectra}. It follows that $A(E) \subseteq \Cstar(g)$ by condition (i) above.

    We now show that all partial isometries associated with interior edges belong to $C^*(g)$. Let $e \in \edgeint$. Then $s_e s_e ^* \in A(E) \subseteq \Cstar(g)$, so $s_e s_e^* g \in \Cstar(g)$. Observe that
    \begin{equation}
    s_e s_e^* g = s_e s_e^* d = \epsilon_e s_e \xi^{\frac{1}{2}}
    \end{equation}
    by orthogonality of range projections. Next, by (ii) above, there exists $h \in A(E) \subseteq \Cstar(g)$ such that $\xi^{\frac{1}{2}}h =s_e^*s_e$. So
    \begin{equation}
    s_e s_e^* g h = \epsilon_e s_e \xi^{\frac{1}{2}}h = \epsilon_e s_e s_e^* s_e = \epsilon_e s_e
    \end{equation}
    is an element of $\Cstar(g)$. Since $\epsilon_e >0$, we get that $s_e \in \Cstar(g)$ as desired.
\end{proof}

\subsection{Boundary edges}\label{subsec:boundar}
The previous subsection proves that all partial isometries associated with interior edges belong to $\Cstar(g)$. We subtract all terms involving interior edges from $g$ to obtain a new element $g_{1,1} \in \Cstar(g)$ defined in \eqref{eq:g11}. All terms involving boundary edges will be captured in a new element $a_1$ defined in \eqref{eq:a_1}. In this subsection we show that all partial isometries associated with boundary edges belong to the \Cstar-subalgebra jointly generated by $g$ and $a_1$. In Subsection~\ref{subsec:sinks} it is shown that $a_1 \in \Cstar(g)$, so that all partial isometries associated with boundary edges belong to $\Cstar(g)$, as desired.

\vspace{0.3cm}
Define
\begin{equation}\label{eq:g11}
g_{1,1} \coloneq g - \left( \sum_{e \in \edgeint} \epsilon_e s_e \right)\xi^{\frac{1}{2}} = a_1 + b + c\,,
\end{equation}
where
\begin{equation}\label{eq:a_1}
    a_1 \coloneq \left( \sum_{y \in \newY} \alpha_{y,1} s_{y,1} \right)\xi^{\frac{1}{2}} + \sum_{\substack{y \in \newY \\ n \ge 1}} \alpha_{y,n+1} s_{y,n+1} s_{y,n}^*\,.
\end{equation}
Alternatively, $a_1$ can be written as $a_1 = d - \left( \sum_{e \in \edgeint} \epsilon_e s_e \right)\xi^{\frac{1}{2}} + a\,$.\footnote{The element $a_1$ is the sum of $a$ and the ``leftover'' terms from $d$.} Note that $g_{1,1} \in C^*(g)$ since $\left( \sum_{e \in \edgeint} \epsilon_e s_e \right)\xi^{\frac{1}{2}} \in \Cstar(g)$ by Proposition~\ref{proposition:interior-edges}.
\begin{lemma}\label{lemma:boundary-edges}
    All partial isometries associated with boundary edges belong to $\Cstar(g, a_1)$.
\end{lemma}

\begin{proof}
    Recall from Notation~\ref{label} that $Y$ denotes the set of vertices $s(\edgebdy) \subseteq \Eint$, and that the set of boundary edges is equal to $\{e_{y,n}\}_{y \in \newY, n \ge 1}$. Our first claim is that for any $y \in \newY$ the partial isometry $s_{y,1}$ belongs to $\Cstar(g, a_1)$. By the construction of $\xi$ and $d$ in \eqref{eq:d}, there exists an element $h \in A(E)$ such that $\xi^{\frac{1}{2}}h= s_{y,1}^* s_{y,1}$. By Lemma~\ref{lemma:orthogonality} we have
    \begin{equation}
    a_1h = \alpha_{y,1} s_{y,1}s_{y,1}^* s_{y,1} = \alpha_{y,1} s_{y,1}\,.
    \end{equation}
    Observe that $a_1h$ belongs to $\Cstar(g, a_1)$ by Proposition~\ref{proposition:interior-edges}, and therefore $s_{y,1}$ also belongs to $\Cstar(g, a_1)$ since $\alpha_{y,1} >0$. This proves the first claim.

    Next, since $\xi \in A(E) \subseteq \Cstar(g)$ (Proposition~\ref{proposition:interior-edges}), we deduce that
    \begin{equation}
    a = \sum_{\substack{y \in \newY \\ n \ge 1}} \alpha_{y,n+1} s_{y,n+1} s_{y,n}^* = a_1 - \left( \sum_{y \in \newY} \alpha_{y,1} s_{y,1} \right)\xi^{\frac{1}{2}}
    \end{equation}
    belongs to $\Cstar(g, a_1)$. Fix $y \in \newY$ and note that ${a s_{y,n} = \alpha_{y,n+1} s_{y,n+1}}$ for any $n \ge 1$ by Lemma~\ref{lemma:orthogonality}. It follows from the first claim that $s_{y,2} \in \Cstar(g, a_1)$ since $\alpha_{y,2} > 0$. Continuing inductively, we deduce that $s_{y,n} \in \Cstar(g, a_1)$ for any $n \ge 1$. Since $y \in \newY$ is arbitrary, we have shown that all partial isometries associated with boundary edges belong to $\Cstar(g, a_1)$, finishing the proof.
\end{proof}

The inequality in the following lemma will be crucial in Lemma~\ref{lemma:proposition-4-part1} and the associated induction procedure in Proposition~\ref{prop:proposition-4-part2}.

\begin{lemma}\label{lemma:a_1}
    Let $a_1$ be as in \eqref{eq:a_1}. Then $a_1 = \sum_{v \in \Eboundary} p_v a_1$ as a norm convergent sum, and the inequality
    \begin{equation}\label{eq:decay-of-a_1}
    \sum_{\ell \ge m} \sum_{v \in V_\ell} \norm{p_v a_1} \le \delta_m
    \end{equation}
    holds for any $m$.
\end{lemma}

\begin{proof}
    For for any $v \in \Eboundary$, we have
    \begin{align}
    \norm{p_v a_1}
    &\le \sum_{y\in \newY} \alpha_{y,1} \norm{p_v s_{y,1}} + \sum_{\substack{y \in \newY \\ n \ge 1}} \alpha_{y,n+1} \norm{p_vs_{y,n+1} s^*_{y,n}}\\
    &= \sum_{\substack{y \in \newY, n \ge 1\\ r(e_{y,n}) = v}} \alpha_{y,n}\,.
    \end{align}
    Therefore, by \eqref{DecompOfa_less_delta} we have
    \begin{equation}
    \sum_{\ell \ge m} \sum_{v \in V_\ell} \norm{p_v a_1} \le \delta_m
    \end{equation}
    for any $m$ as desired. Taking $m = 1$ gives that $\sum_{v \in \Eboundary} \norm{p_v a_1} \le \delta_1 < \infty$. Hence the sum $\sum_{v \in \Eboundary} p_v a_1$ converges absolutely in norm. Since all edges in $\{e_{y,n}\}_{y \in \newY, n \ge 1}$ have range in $\Eboundary$, it is clear from \eqref{eq:a_1} that this sum is equal to $a_1$.
\end{proof}

\subsection{Sinks and sink edges}\label{subsec:sinks}
In this subsection we show that $\Cstar(g)$ contains all sink projections and all partial isometries associated with sink edges. It will then follow that the candidate generator $g$ from \eqref{generator} does indeed generate $\Cstar(E)$.

In Proposition~\ref{prop:proposition-4-part2}, we show by induction (with respect to the lexicographic ordering on $\N \times \N$) that all sink edge partial isometries $t_{M,N}$ belong to $\Cstar(g)$. Lemma~\ref{lemma:proposition-4-part1} provides the inductive step. Simultaneously, we show by induction on $M$ that all sink projections $q_M$ belong to $\Cstar(g)$. Once we have shown that a given partial isometry $t_{M,N}$ belongs to $\Cstar(g)$, we subtract all terms involving $t_{M,N}$ from the generator to obtain a ``modified generator'' $g_{M,N+1}$ (see \eqref{eq:g_NM}). Similarly, once we have shown that a given sink projection $q_M$ belongs to $\Cstar(g)$, we subtract all terms involving $q_{M}$ from the generator to obtain a ``modified generator'' $g_{M+1,1}$. Note that this also includes subtracting all terms involving boundary edges with range in $V_M$, so that $a_M$ is replaced by $a_{M+1}$ (see \eqref{eq:a_M}). 

\vspace{0.3cm}
We equip $\N \times \N$ with the lexicographic ordering. That is, we write $(m_1,n_1) \ge (m_2, n_2)$ if and only if $m_1 > m_2$ or $n_1 \ge n_2$ and $m_1 = m_2$. For $N,M \ge 1$ we define
\begin{equation}\label{eq:g_NM}
g_{M,N} \coloneq  a_M + \sum_{(m,n) \ge (M,N)} \beta_{m,n} t_{m,n} +  \sum_{m \ge M} \delta_m q_m + \sum_{(m,n) \ge (M,N)} (\gamma_{m,n}-\delta_m) t_{m,n}t_{m,n}^*\,,
\end{equation}
where
\begin{equation}\label{eq:a_M}
    a_M \coloneq \sum_{\ell \ge M}\sum_{v \in V_\ell}p_v a_1\,.
\end{equation}
We have used Lemma~\ref{lemma:a_1} to show that the sum in \eqref{eq:a_M} is absolutely convergent. Note that this notation agrees with $g_{1,1}$ and $a_1$ as defined in \eqref{eq:g11} and \eqref{eq:a_1}.

\begin{lemma}
\label{lemma:proposition-4-part1}
    For any $M,N \ge 1$ we have $t_{M,N} \in \Cstar(g_{M,N})$ and $q_{M} \in \Cstar(g_{M+1,1} + \delta_M q_M)$.
\end{lemma}

\begin{proof}

    For $k \ge 1$, we have the following by Lemma~\ref{lemma:orthogonality}:
    \begin{align}\label{eq:bigsumbigk}
    \frac{1}{\gamma_{M,N}^k}g_{M,N}^k &= \ 
    \frac{1}{\gamma_{M,N}^k}a_M^k 
    + \frac{1}{\gamma_{M,N}^k} \sum_{m \ge M} \delta_m^k q_m + \sum_{(m,n) \ge (M,N)} \frac{\gamma_{m,n}^k-\delta_m^k}{\gamma_{M,N}^k} t_{m,n}t_{m,n}^*\\
    &+ \sum_{(m,n) \ge (M,N)} \frac{\beta_{m,n} t_{m,n} }{\gamma_{M,N}^k} \left( \sum_{j=1}^{k} \gamma_{m,n}^{k-j} \ a_M^{j-1} \right).
    \end{align}
    By assumption we have $\delta_M < \gamma_{M,N}$, and $\norm{a_M} \le \delta_M$ by Lemma~\ref{lemma:a_1}, so the first term converges to $0$ as $k \to \infty$. The second term converges to 0 by the monotone convergence theorem, since $\delta_m \le \delta_M < \gamma_{M,N}$ for all $m \ge M$. For the third and fourth terms, the sums over $(m,n) \ge (M, N+1)$ will converge to 0 as $k \to \infty$. Indeed
    \begin{equation}
    \norm{\sum_{(m,n) \ge (M,N+1)} \frac{\gamma_{m,n}^k-\delta_m^k}{\gamma_{M,N}^k} t_{m,n}t_{m,n}^*}
    \le \max_{(m,n) \ge (M,N+1)} \left( \frac{\gamma_{m,n}^k-\delta_m^k}{\gamma_{M,N}^k}\right)
    \le \frac{\gamma_{M,N+1}^k}{\gamma_{M,N}^k}
    \end{equation}
    since the projections $t_{m,n}t_{m,n}^*$ are mutually orthogonal and the double sequence $(\gamma_{m,n})$ is decreasing with respect to the lexicographic ordering. Note that $\sum_{j=1}^\infty \gamma_{M,N+1}^{-j} \norm{a_M}^{j-1}$ converges since $\delta_M < \gamma_{M,N+1}$. Then for $(m,n) \ge (M,N+1)$ we have
    \begin{align}
    \norm{\frac{\beta_{m,n} t_{m,n}}{\gamma_{M,N}^k} \sum_{j=1}^{k} \gamma_{m,n}^{k-j} \ a_M^{j-1}}
    &\le \frac{\beta_{m,n} \gamma_{M,N+1}^k}{\gamma_{M,N}^k} \sum_{j=1}^\infty \gamma_{M,N+1}^{-j} \norm{a_M}^{j-1}\\
    &\le C_{M,N} \beta_{m,n} \left(\frac{\gamma_{M,N+1}}{\gamma_{M,N}}\right)^k,
    \end{align}
    where $C_{M,N}$ is a constant independent of $m,n$ and $k$. It is then clear that the sum of these terms over $(m,n) \ge (M,N+1)$ converges to 0 as $k \to \infty$. 
    Overall, when taking the limit $k \to \infty$ in \eqref{eq:bigsumbigk}, all that remains is the limit of the summands $(m,n) = (M,N)$ from the third and fourth terms. Therefore, we get
    \begin{equation}\label{eq:converge-to-y}
    \frac{1}{\gamma_{M,N}^k}g_{M,N}^k \to t_{M,N}t_{M,N}^* + \beta_{M,N} t_{M,N} \sum_{j=1}^\infty \gamma_{M,N}^{-j} a_M^{j-1} \eqqcolon y
    \end{equation}
    as $k \to \infty$, so $y \in C^*(g_{M,N})$. Write
    \begin{equation}
    z \coloneq \beta_{M,N} \sum_{j=1}^\infty \gamma_{M,N}^{-j} a_M^{j-1}\,,
    \end{equation}
    so $y = t_{M,N}t_{M,N}^* + t_{M,N} z$. Then
    $
    yy^* = t_{M,N}t_{M,N}^* + t_{M,N} z z^* t_{M,N}^*\,.
    $
    Since $\norm{a_M} \le \delta_M$ and $\beta_{M,N} = \frac{1}{2} (\gamma_{M,N} - \delta_M)$, we get that
    \begin{equation}
    \norm{z} \leq \frac{\beta_{M,N}}{\gamma_{M,N} - \norm{a_M}} \leq \frac{1}{2}\,.
    \end{equation}
    Hence $yy^*$ is a small perturbation of the projection $t_{M,N}t_{M,N}^*$, so that
    $
    \lim_{k\to\infty} (yy^*)^{1/k} = t_{M,N}t_{M,N}^*\,.
    $\footnote{Explicitly, $yy^*$ is the sum of a projection $t_{M,N}t_{M,N}^*$ and a positive element $x \coloneq t_{M,N} z z^* t_{M,N}^*$ that is contained in $t_{M,N}t_{M,N}^*$ and has norm strictly less than 1. By working in the commutative $C^*$-algebra $C^*(x, t_{M,N}t_{M,N}^*)$, it is clear that $(x + t_{M,N}t_{M,N}^*)^{1/k} \to t_{M,N}t_{M,N}^*$ as $k \to \infty$.} Thus $t_{M,N}t_{M,N}^* \in C^*(g_{M,N})$. Finally, note that
    \begin{equation}
    t_{M,N}t_{M,N}^* g_{M,N} = \beta_{M,N} t_{M,N} + \gamma_{M,N} t_{M,N}t_{M,N}^*\,
    \end{equation}
    so $t_{M,N} \in C^*(g_{M,N})$, as claimed.

    Next, consider the limit of $\frac{1}{\delta_M^k}\big(g_{M+1,1} + \delta_M q_M\big)^k$ as $k \to \infty$. The elements $g_{M+1,1}$ and $q_M$ are orthogonal by Lemma~\ref{lemma:orthogonality}, hence $\big(g_{M+1,1} + \delta_M q_M\big)^k = g_{M+1,1}^k + \delta_M^k q_M$. By \eqref{eq:converge-to-y} the limit $\lim_{k \to \infty}\frac{1}{\gamma_{M+1,1}^k}g_{M+1,1}^k$ exists, hence $\lim_{k \to \infty}\frac{1}{\delta_{M}^k}g_{M+1,1}^k = 0$ since $\gamma_{M+1,1} < \delta_M$. It follows that $\lim_{k \to \infty} \frac{1}{\delta_M^k}\big(g_{M+1,1} + \delta_M q_M\big)^k = q_M$. So indeed $q_m \in \Cstar(g_{M+1,1} + \delta_M q_M)$, as claimed.
\end{proof}

\begin{prop}
\label{prop:proposition-4-part2}
    Let $g_{1,1}$ be as in \eqref{eq:g11}. All projections and partial isometries associated with sinks and sink edges belong to $\Cstar(g_{1,1})$, and hence to $\Cstar(g)$. 
    Also, $a_1$ belongs to $C^*(g_{1,1})$ and hence to $\Cstar(g)$, where $a_1$ is as in \eqref{eq:a_1}.
\end{prop}

\begin{proof}    
    Fix $m \ge 1$. Let us show, by induction on $n$, that $t_{m,n} \in \Cstar(g_{m,1})$ for all $n \ge 1$. By Lemma~\ref{lemma:proposition-4-part1}, the partial isometry $t_{m,1}$ belongs to $\Cstar(g_{m,1})$. Assume that ${t_{m,1}, \dots, t_{m,N} \in \Cstar(g_{m,1})}$ for some $N \ge 1$. Then
    \begin{equation}
    g_{m,N+1} = g_{m,1} - \sum_{n=1}^N \beta_{m,n} t_{m,n} - \sum_{n=1}^N (\gamma_{m,n} - \delta_{m}) t_{m,n} t_{m,n}^*
    \end{equation}
    also belongs to $\Cstar(g_{m,1})$. By Lemma~\ref{lemma:proposition-4-part1}, the partial isometry $t_{m,N+1}$ belongs to $\Cstar(g_{m,N+1})$ and hence to $\Cstar(g_{m,1})$. Therefore, by induction on $n$ we have that $t_{m,n} \in \Cstar(g_{m,1})$ for all $n \ge 1$, as claimed.

    Next let us show, by induction on $m$, that $q_m$ and $g_{m,1}$ belong to $\Cstar(g_{1,1})$ for any $m \ge 1$. Assume that $g_{m,1} \in \Cstar(g_{1,1})$ for some $m \ge 1$. By the claim above, the partial isometry $t_f$ belongs to $C^*(g_{1,1})$ for any edge $f$ terminating at the sink $w_m$. By considering $t_f^* t_f$ for such edges, it follows that $p_v$ belongs to $C^*(g_{1,1})$ for any vertex $v \in V_m$ (since for any vertex $v \in V_m$ there exists an edge whose source is equal to $v$ and whose range is equal to $w_m$). 
    By Lemma~\ref{lemma:orthogonality} we have $p_v g_{1,1} = p_v a_1$ for any $v \in \Eboundary$ and therefore $\sum_{v \in V_m} p_v a_1 \in C^*(g_{1,1})$, where we have used Lemma~\ref{lemma:a_1} to show that this sum converges. By the first claim, it then follows that
    \begin{equation}
    g_{m+1,1} + \delta_m q_m = g_{m,1} - \sum_{n\ge 1} \beta_{m,n} t_{m,n} - \sum_{n \ge 1} (\gamma_{m,n} - \delta_{m}) t_{m,n} t_{m,n}^* -\sum_{v \in V_m} p_v a_1
    \end{equation}
    also belongs to $\Cstar(g_{1,1})$. By Lemma~\ref{lemma:proposition-4-part1}, the projection $q_m$ belongs to $\Cstar(g_{m+1,1} + \delta_m q_m)$ and hence to $\Cstar(g_{1,1})$. Subtracting $\delta_m q_m$ from the expression above shows that $g_{m+1,1}$ belongs to $\Cstar(g_{1,1})$. Therefore, by induction on $m$ the elements $q_m$ and $g_{m,1}$ belong to $C^*(g_{1,1})$ for any $m \ge 1$. By the first claim, it follows that $t_{m,n} \in C^*(g_{1,1})$ for all $m,n \ge 1$.
    
    We have shown that for any sink edge $f$, the partial isometry $t_f$ lies in $C^*(g_{1,1})$. Therefore, by considering $t_f^* t_f$ for such edges, it follows that $p_v$ belongs to $C^*(g_{1,1})$ for any boundary vertex $v$. By Lemma~\ref{lemma:orthogonality} we have $p_v g_{1,1} = p_v a_1 \in C^*(g_{1,1})$ for any $v \in \Eboundary$, and therefore $a_1$ belongs to $C^*(g_{1,1})$ since $a_1 = \sum_{v \in \Eboundary} p_v a_1$ (Lemma~\ref{lemma:a_1}). This completes the proof of the proposition.
\end{proof}

\begin{thm}
    The element $g$ from \eqref{generator} generates $\Cstar(E)$ as a \Cstar-algebra.
    Consequently, the \Cstar-algebra of any countable directed graph is singly generated.
\end{thm}

\begin{proof}
    By Proposition~\ref{proposition:interior-edges}, all partial isometries associated with interior edges belong to $\Cstar(g)$. By Proposition~\ref{prop:proposition-4-part2}, all projections and partial isometries associated with sinks and sink edges belong to $\Cstar(g)$. Also, the element $a_1$ defined in \eqref{eq:a_1} belongs to $\Cstar(g)$ by Proposition~\ref{prop:proposition-4-part2}, and therefore all partial isometries associated with boundary edges belong to $\Cstar(g)$ by Lemma~\ref{lemma:boundary-edges}.
    To conclude, recall that every graph \Cstar-algebra is generated by its edge partial isometries and sink projections.
\end{proof}

\begin{remark}
    As hinted at in the introduction, the argument can be simplified considerably when no sink receives infinitely many edges or when there are no sinks at all. 

    The latter is case is handled in Corollary~\ref{corollary:no-sinks}. In this case, all edges are interior edges, and the generator $g$ \eqref{generator} is equal to $d$ \eqref{eq:d}.
    
    If the graph $E$ has no sinks receiving infinitely many edges, then the proof also simplifies since the \Cstar-algebra is generated only by the edge partial isometries in this case.
    More concretely, Subsection~\ref{subsec:sinks} becomes redundant, and an element of the form $\tilde{a}+\tilde{d}+\sum_{n} \frac{1}{n^2}q_n$ generates $\Cstar(E)$, where $\{q_1, q_2, \dots\}$ is the set of projections corresponding to vertices that don't emit or receive any edges, and where $\tilde{a}$ and $\tilde{d}$ are as defined in \eqref{eq:a} and \eqref{eq:d}, but with the modified definitions of boundary and interior edges described in the following. In this setting, one deletes the sets $\Esink$ and $\edgesink$, defines an edge $e \in E^1$ to be a \emph{boundary edge} if it terminates in a sink, and defines an edge to be an \emph{interior edge} otherwise.
\end{remark}

\bibliographystyle{abbrv}
\bibliography{citations}

\end{document}